\documentclass[12pt]{article}
\usepackage{amsmath,amssymb,amsthm}

\newtheorem{theorem}{Theorem}[section]
\newtheorem{lemma}[theorem]{Lemma}
\newtheorem{definition}[theorem]{Definition}

\newtheorem{example}[theorem]{Example}

\newtheorem{conjecture}[theorem]{Conjecture}

\newcommand\Z{\mathbb{Z}}

\newcommand{\PP}{{\mathbb{P}}}
\newcommand{\NN}{{\mathbb{N}}}
\newcommand{\ZZ}{{\mathbb{Z}}}

\newcommand{\EE}{{\mathbb{E}}}
\newcommand{\SSS}{{\mathbb{S}}}

\begin{document}

\title{The Cauchy-Davenport Theorem for Finite Groups}
\author{Jeffrey Paul Wheeler\thanks{Department of Mathematics,
the University of Pittsburgh, Pittsburgh PA 15260, USA.
E-mail: \texttt{jwheeler@pitt.edu}.}}
\date{February 2006}
\maketitle

\begin{abstract}
The Cauchy-Davenport theorem states that for any two nonempty
subsets $A$ and $B$ of $\Z/p\Z$ we have $|A+B|\ge \min\{p,|A|+|B|-1\}$,
where $A+B:=\{a+b\bmod p\mid a\in A,\, b\in B\}$.
We generalize this result from $\Z/p\Z$ to arbitrary finite
(including non-abelian) groups. This result from early in $2006$ is
independent of Gyula K\'{a}rolyi's\footnote{Gyula was a visitor at the University of Memphis early in my time as a graduate student there.  I wish to thank
him for introducing me to this problem and encouraging me to work on it (in addition to teaching a great class on the subject).  Regrettably I did not inform Gyula that I was working on the problem and that progress was being made, hence the independent results. I discovered Gyula had a result the day before presenting the work to the Combinatorics seminar at the University of Memphis (which was incredibly disappointing to a graduate student with his first result).}
$2005$ result in \cite{Gyula3} and uses different methods.
\end{abstract}

\section{Background and Motivation}

The problem we will be considering lies in the area of Additive
Number Theory. This relatively young area of Mathematics is part of
Combinatorial Number Theory and can best be described as the study
of sums of sets of integers. As such, we begin by stating the
following definition:
\begin{definition}\label{definition:A+B}
 For subsets $A$ and $B$ of a group $G$, define
 \[
  A+B:=\{a+b\mid a\in A,\, b\in B\}
 \]
 where $+$ is the group operation. We write
 \[
  A\cdot B:=\{ab\mid a\in A,\,b\in B\}
 \]
 in the case when the group $G$ is written multiplicatively.
\end{definition}

A simple example of a problem in Additive Number Theory is given two
subsets $A$ and $B$ of a set of integers, what facts can we
determine about the sumset $A+B:=\{a+b \mid a \in A, b \in B \}$?  The topic of this paper is one such problem.  Note that a very familiar problem in Number Theory, namely
Lagrange's theorem\index{Four Square Theorem} that every nonnegative
integer can be written as the sum of four squares, can be expressed
in terms of sumsets. In particular,

\begin{theorem}\index{Lagrange's Four Square
Theorem}[Lagrange's Four Square Theorem]\ \\
Let $\NN_0 = \{x \in \ZZ \mid x \geq 0 \}$ and let $\SSS = \{ x^2
\mid x \in \ZZ \}$.  Then
$$\NN_0 = \SSS + \SSS + \SSS + \SSS.$$
\end{theorem}

As well the the binary version of Goldbach's Conjecture can be
restated in terms of sumsets.
\begin{conjecture}\index{Goldbach's Conjecture}[Goldbach's Conjecture]\
\\
Let $\EE = \{ 2x \mid x \in \ZZ, x \geq 2 \}$ and let $\PP = \{ p
\in \ZZ \mid p$ is prime $\}$.  Then
\begin{align}
\EE &\subseteq \PP + \PP.\label{goldbach}
\end{align}
\end{conjecture}

In other words, every even integer is greater than $2$ is the sum of
two primes.  Notice that we do not have set equality in equation
$(\ref{goldbach})$ because $2 \in \PP$.

The theorem we wish to extend was first proved by Augustin Cauchy in
$1813$\footnote{Cauchy used this theorem to prove that $Ax^2 + By^2
+ C \equiv 0 (\bmod p)$ has solutions provided that $ABC \not \equiv 0$.
  This is interesting in that Lagrange used
this result to establish his four squares theorem.}
\cite{Cauchy} and later independently reproved by Harold Davenport
in $1935$ \cite{Davenport} (Davenport discovered in $1947$
\cite{DavenportHist} that Cauchy had previously proved the theorem.)
In particular,
\begin{theorem}[Cauchy-Davenport]\label{theorem:cauchy-davenport}
 If\/ $A$ and\/ $B$ are nonempty subsets of $\Z/p\Z$, $p$ prime,
 then $|A+B|\ge\min\{p,|A|+|B|-1\}$.
\end{theorem}

We note that in $1935$ Inder Chowla \cite{Chowla} extended the
result to composite moduli $m$ when $0 \in B$ and the other members
of $B$ are relatively prime to~$m$.  As well it is worth noting that
in $1996$ Alon, Nathanson, and Ruzsa provided a simple proof of this
theorem using the Polynomial Method~\cite{Alon}.

Of interest to this work is Gyula K\'{a}rolyi's extension of the
theorem to abelian groups \cite{Gyula1,Gyula2}. Before we
state the theorem, though, a useful definition:
\begin{definition}[Minimal Torsion Element]\label{definition:minimal torsion element}
 Let $G$ be a group.  We
 define $p(G)$ to be the smallest positive integer $p$ for which
 there exists a nonzero element $g$ of $G$ with $pg=0$ (or, if
 multiplicative notation is used, $g^{p}=1$, $g\ne1$). If no such $p$ exists,
 we write $p(G)=\infty$.
\end{definition}

\begin{lemma}\label{lemma:p(G)}
 For any finite group $G\ne\{1\}$, $p(G)$ is the smallest prime factor of\/~$|G|$.
\end{lemma}
\begin{proof}
Let $p$ be the smallest prime dividing $|G|$.
Then by Cauchy's Theorem, there is an
element $g \in G$ or order~$p$, i.e., $g^p=1$ but $g\ne1$.
Suppose there were a smaller prime $q$ with $h^q=1$, $h\ne1$.
Then $|\langle h\rangle|=q$ and by Lagrange's
Theorem $q\mid|G|$. This contradicts the choice of~$p$.
\end{proof}

Now we state the generalization of Theorem~\ref{theorem:cauchy-davenport}
to abelian groups.
\begin{theorem}[K\'arolyi \cite{Gyula1,Gyula2}]\label{theorem:abelianCD}
 If\/ $A$ and\/ $B$ are nonempty subsets of an abelian group $G$, then
 $|A+B| \ge \min\{p(G),|A|+|B|-1\}$.
\end{theorem}
Before we continue we state a famous and very useful result.
\begin{theorem}[Feit-Thompson \cite{FT}]\label{theorem:feit-thompson}
Every group of odd order is solvable.
\end{theorem}
Since any group $G$ of even order has $p(G)=2$, we will mainly be
considering groups of odd order. Hence by
Theorem~\ref{theorem:feit-thompson}, we will mainly be considering
only solvable groups.

\section{A Basic Structure of Finite Solvable Groups}

Throughout this section $G$ will be a finite solvable group, i.e.,
there exists a chain of subgroups
\[
 \{1\}=G_0 \unlhd G_1 \unlhd G_2\unlhd \cdots \unlhd G_n = G
\]
such that $G_{i-1}$ is a normal subgroup of $G_i$
and the quotient group $G_i/G_{i-1}$ is abelian for $i=1,2,\dots,n$.

Hence by definition, either $G=\{1\}$, or there is some proper
normal subgroup $K\lhd G$ such that $K$ is also solvable
and the quotient group $G/K$ is abelian. (In fact for finite
groups one can also insist that $G/K$ is cyclic of prime order,
i.e., isomorphic to some $\Z/p\Z$, $p$ prime. However we shall
not require this.)

Fix for each coset $h\in G/K$ a representative $\tilde{h}\in G$,
so that $\tilde{h}\in h$ and $h=K\tilde{h}\in G/K$. Each $g \in G$
lies in a unique coset $h\in G/K$, and then $g\tilde{h}^{-1}$
lies in~$K$. Thus there is a $k\in K$ and an $h\in G/K$
such that $g=k\tilde{h}$. The pair $(k,h)$ is unique since if
$g=k\tilde{h}=k'\tilde{h'}$ then $Kg=h=h'$ and then
$k=g\tilde{h}^{-1}=k'$.
Define
\begin{equation}\label{psi}
 \psi\colon G \rightarrow K \times G/K
 \quad\text{by } \psi(g)= (k, h),\quad\text{where }g=k\tilde{h}.
\end{equation}
Then $\psi$ is a bijection between the sets $G$ and $K\times G/K$.
Define an operation $\star$ on $K \times G/K$ by
\[
(k_1,h_1) \star (k_2,h_2):=(k_1\phi_{h_1}(k_2)\eta_{h_1,h_2},h_1 h_2).
\]
where $\phi_h\in\operatorname{Aut}(K)$
is defined by $\phi_h(k)=\tilde{h}k\tilde{h}^{-1}\in K$
(recall $K\unlhd G$), and
\begin{equation}\label{eta}
 \eta_{h_1,h_2}=\tilde{h}_1\cdot\tilde{h}_2\cdot(\widetilde{h_1h_2})^{-1}\in K.
\end{equation}
Note that both $\phi_h$ and $\eta_{h_1,h_2}$ depend on the choice of coset
representatives of $G/K$.
As will be seen in the examples, $\eta$ plays a role analogous to
``carrying the one'' in simple arithmetic.

\begin{lemma}[Basic Structure of Solvable Groups]\label{lemma:basic structure}
 Let\/ $G$ be a solvable group with\/ $K \unlhd G$. Upon fixing the
 set\/ $R=\{\tilde{h}\mid h\in G/K\}$ of coset representatives of\/~$G/K$,
 $\psi$ in \eqref{psi} is an isomorphism from\/ $G$ to the group\/ $(K\times G/K,\star)$.
\end{lemma}
\begin{proof}
As noted above, $\psi$ is a bijection, so it is enough to show
it is a homomorphism. Suppose that $g_1=k_1\tilde{h}_1$ and $g_2=k_2\tilde{h}_2$.
Then
\begin{align*}
 \psi(g_1) \star \psi(g_2) &= (k_1,h_1) \star (k_2,h_2)\\
 &=(k_1 \phi_{h_1}(k_2) \eta_{h_1, h_2}, h_1 h_2)\\
 &=(k_1 \tilde{h}_1 k_2 \tilde{h}_1^{-1} \tilde{h}_1 \tilde{h}_2
 (\widetilde{h_1 h_2})^{-1}, h_1 h_2)\\
 &=\psi(k_1\tilde{h}_1 k_2 \tilde{h}_2
 (\widetilde{h_1 h_2})^{-1} \widetilde{h_1 h_2})\\
 &= \psi(k_1\tilde{h}_1 k_2\tilde{h}_2)\\
 &= \psi(g_1 g_2)
\end{align*}
\end{proof}

%In summary, for $A \subseteq G$, we can represent $A$ as a subset
%$\{(k_1,h_1),\dots,(k_t,h_t)\}$ of $K\times G/K$
%where $k_1,k_2,\dots,k_t\in K$ and $h_1,h_2,\dots h_t \in G/K$.
%We note that it is certainly not the case that the $k_i$'s nor
%the $h_j$'s are distinct.

It is worth noting that the construction of $\star$ on $K \times G/K$
is more general than the semi-direct product of two groups. Indeed, $G$ may
not be a semi-direct product of $K$ and $G/K$. If however it is, then one
can choose the representatives $\tilde{h}$ so that $\eta_{h_1,h_2}=1$
for all $h_1,h_2\in G/K$.

Before we continue, consider two illustrative examples.

\begin{example} Let $p$ be a prime, $G=\Z/p^2\Z$ and $K=p\Z/p^2\Z$.
Let the representatives of $G/K$ be $\{0,1,\dots,p-1\}$ ($\bmod p^2$).
Then we can represent
$G$ as a set of pairs $(ap,b+K)$, $a,b\in\{0,1,\dots,p-1\}$
(or more simply as just $(a,b)$, see Table~\ref{zp2}).
The automorphism $\phi_{b+K}$ is the identity as
$G$ is abelian. However, $\eta(b+K,d+K)=(b\bmod p)+(d\bmod p)-(a+b\bmod p)$
which is $0\in K$ when $b+d<p$ and $p=1p\in K$ when $b+c\ge p$.
Hence addition in $G$ is given by $(a,b)+(c,d)=(a+b+\eta\bmod p,b+d\bmod p)$
where $\eta=0$ if $b+d<p$ and $1$ if $b+d\ge p$.
However, this is effectively just ``addition with carry'' of two digit base~$p$
numbers.

\begin{table}[!h] %the [h] tells LaTeX to put the table "right here" with emphasis
 \caption{Elements $g\in\Z/p^2\Z$ represented as pairs $\phi(g)=(a,b)$, $a,b\in\{0,\dots,p-1\}$.}
 \label{zp2}
 \[\begin{array}{|c|cccccccc|}\hline
  g&0&1&\cdots\!&p-1&p&p+1&\cdots\!&p^2-1\\\hline
  \psi(g)&(0,0)&(0,1)&\cdots\!&(0,p\!-\!1)&(1,0)&(1,1)&\cdots\!&(p\!-\!1,p\!-\!1)\\\hline
 \end{array}\]
\end{table}
\end{example}

\begin{example} Let $Q$ be the quaternion group, namely
$Q=\{ \pm 1, \pm i, \pm j, \pm k \}$ where
$ij=k$, $jk=i$, $ki=j$, $ji=-k$, $kj=-i$, $ik=-j$ and $i^2=j^2=k^2=-1$.
Put $K=\{\pm 1,\pm k\}$, so that $Q/K=\{K,Kj\}$ and we choose $1$ as our
coset representative of $K$ and $j$ as the coset representative of $Kj$

The order of $Kj$ in $Q/K$ is $2$ however the order of $j$ in
$Q$ is $4$. Indeed
\[
 \eta_{Kj,Kj}=j\cdot j\cdot 1^{-1}=j^2=-1
\]
since the coset representative of $(Kj)^2=K$ is~$1$.

Thus, since $i=-k \cdot j$,
\begin{align*}
 \psi(i\cdot i)
 &=(-k,Kj)(-k,Kj)\\
 &=(-k (j(-k)j^{-1})\eta_{j,j} \},(Kj)^2)\\
 &=(-k(k)(-1),K)\\
 &=(-1,K)=\psi(-1)
\end{align*}
Which is what we hoped for since $i \cdot i = -1$.
\end{example}

\begin{table}[!h] %the [h] tells LaTeX to put the table "right here" with emphasis
 \caption{Elements of $Q$ written as pairs with coset representatives 1 and $j$.}\label{table:q=(k,h)}
 \[\begin{array}{|c|c@{\ }c@{\ }c@{\ }c@{\ }c@{\ }c@{\ }c@{\ }c|}\hline
  g\in Q & 1 & -1 & i & -i & j & -j & k & -k \\\hline
 \psi(g)&(1,\!K)&(\!-1,\!K)&(\!-k,\!Kj)&(k,\!Kj)&
 (1,\!Kj)&(\!-1,\!Kj)&(k,\!K)&(\!-k,\!K)\\\hline
 \end{array}\]
\end{table}

This basic structure is sufficient for tackling the job of establishing
the Cauchy-Davenport Theorem for finite groups.  It is worth mentioning
that a more sophisticated structure for solvable groups is required for
proving the related problem of Erd\H{o}s and Heilbronn
(see~\cite{Balister/Wheeler}).
As well, we note that neither $\Z/p^2\Z$ nor the quaternion group is a
semidirect product of its respective $K$ and $G/K$.

Before proceeding, developing some notation will be helpful.

\begin{definition}\label{definition:S}
Let $G$ be a finite solvable group and $K\lhd G$.
For $S \subseteq G$, represent $S$ as a subset $\psi(S)=\{\psi(g)\mid g\in S\}$
of $K\times G/K$ as above and write
\begin{align*}
 S^1 &:= \{k\in K\mid \exists h\in G/K\colon (k,h)\in \psi(S)\},\\
 S^2 &:= \{h\in G/K\mid \exists k\in K\colon (k,h)\in \psi(S)\}.
\end{align*}
In other words, $S^1$ is the collection of first coordinates
and $S^2$ is the collection of second coordinates of the elements
of $\psi(S)$.
\end{definition}

\section{The Cauchy-Davenport Theorem for Finite Solvable Groups}

\begin{theorem}\label{theorem:CDsolvable}
 Suppose\/ $G$ is a finite solvable group and\/ $A$, $B$ are non-empty
 subsets of\/~$G$. If\/ $|A|+|B|-1\le p(G)$, then\/ $|A\cdot B|\ge|A|+|B|-1$.
\end{theorem}
\begin{proof}
We will proceed by induction on $|G|$. The result is trivial for
$|G|=1$ (note $p(\{1\})=\infty$), so assume $|G|>1$.
As previously stated, there exists a
$K\lhd G$ so that $G/K$ is abelian, $K\ne G$, and $K$ is solvable.

Let $A$ and $B$ be non-empty subsets of $G$. Write $A$ and
$B$ as subsets of $(K,G/K)$ as above.
Let $A^2=\{h_1,\dots,h_\alpha\}$ and $B^2=\{h'_1,\dots,h'_\beta\}$.
For $i=1,\dots,\alpha$, write
\[
 A_i=\{(k,h)\in \psi(A)\mid h=h_i\},\qquad a_i=|A_i|
\]
and similarly for $B_j$, $b_j=|B_j|$, $j=1,\dots,\beta$.
Assume the $h_i$ and $h'_j$ are ordered so that
\[
 a_1\ge a_2\ge\dots\ge a_\alpha,\qquad b_1\ge b_2\ge\dots\ge b_\beta.
\]
We shall also assume without loss of generality that $\alpha\le\beta$.
Note that $\psi(A)=\bigcup_i A_i$, $\psi(B)=\bigcup_j B_j$, and
\[
 |A|=a_1+a_2+\dots+a_\alpha,\qquad |B|=b_1+b_2+\dots+b_\beta.
\]

If $(k,h_i)\in A_i$ and $(k',h'_j)\in B_j$, then
\[
 (k,h_i)\star(k',h'_j)=(k\phi_{h_i}(k')\eta_{h_i,h'_j},h_ih'_j)
\]
Thus, as $h_i$ and $h'_j$ are constant for all elements in $A_i$
and $B_j$ respectively,
\[
 |A_i\cdot B_j|=|(A_i\cdot B_j)^1|=|A_i^1\cdot B'_j|
\]
where $B'_j=\{\phi_{h_i}(k')\eta_{h_i,h'_j}\mid (k',h'_j)\in B_j\}\subseteq K$.
Note that $|B'_j|=|B_j|$ as $\phi_{h_i}$ is an automorphism of $K$ and
multiplication by the constant $\eta_{h_i,h'_j}$ does not change the size
of the set. But $K$ is solvable and $|G|=|K||G/K|$ so
$a_i+b_j-1\le |A|+|B|-1\le p(G)\le p(K)$. Hence by induction,
\[
 |A_i\cdot B_j|=|A_i^1\cdot B'_j|\ge a_i+b_j-1.
\]

Now $A^2,B^2\subseteq G/K$ and $|G|=|K||G/K|$, so
$\alpha+\beta-1\le |A|+|B|-1\le p(G)\le p(G/K)$.
Hence by Theorem~\ref{theorem:abelianCD} (or by Theorem~\ref{theorem:cauchy-davenport}
if we insist that $G/K$ is cyclic of prime order),
\[
 |A^2\cdot B^2|=|\{h_ih'_j\mid 1\le i\le\alpha,\,1\le j\le\beta\}|\ge\alpha+\beta-1.
\]
Now $\alpha+\beta-1 = (\beta) + (\alpha - 1)$, so there are at
least $\alpha-1$ elements $h_ih'_j$ that are not one of the $\beta$
distinct elements $h_1h'_1, h_1h'_2,\dots,h_1h'_\beta\in G/K$.
In particular, $|A\cdot B|$ contains at least $\alpha-1$ elements
that are not in any $A_1\cdot B_j$. Since the second coordinate of every
element of $A_1\cdot B_j$ is $h_1h'_j$, the sets $A_1\cdot B_j$ are disjoint.
Thus
\begin{align*}
 |A\cdot B| &\ge|A_1\cdot B_1|+|A_1\cdot B_2|+\dots+|A_1\cdot B_\beta|+\alpha-1\\
 &\ge \sum_{j=1}^\beta(a_1+b_j-1)+\alpha-1\\
 &= \beta a_1+|B|-\beta+\alpha-1\\
 &\ge |A|+B|-1,
\end{align*}
where in the last line we have used the fact that $\beta\ge\alpha$, and
$\beta a_1\ge \alpha a_1=\sum_{i=1}^\alpha a_1\ge\sum_{i=1}^\alpha a_i=|A|$.
\end{proof}

\section{The Cauchy-Davenport Theorem for Finite Groups}

We now extend Theorem~\ref{theorem:CDsolvable} to all finite
groups.

\begin{theorem}
 Let\/ $G$ be a finite group and let\/ $A$ and\/ $B$ be
 non-empty subsets of\/~$G$.
 Then\/ $|A\cdot B|\ge\min\{p(G),|A|+|B|-1\}$.
\end{theorem}
\begin{proof}
If $G$ is of even order then $p(G)=2$. The result is then trivial
as $|A\cdot B|\ge 2=p(G)$ if either $|A|>1$ or $|B|>1$, while
$|A\cdot B|=1=|A|+|B|-1$ if $|A|=|B|=1$. If $G$ is of odd order
then by Theorem~\ref{theorem:feit-thompson}, $G$ is solvable.  The result
then follows from Theorem~\ref{theorem:CDsolvable} when
$|A|+|B|-1\le p(G)$. If $|A|+|B|-1>p(G)$, take non-empty
subsets $A^*\subseteq A$, $B^*\subseteq B$ such that $|A^*|+|B^*|-1=p(G)$.
Then $|A\cdot B|\ge|A^*\cdot B^*|=p(G)$.
\end{proof}

\section{Closing Remarks}
A problem closely related to the Cauchy-Davenport Theorem was the conjecture Paul Erd\H{o}s and Hans Heilbronn posed in the early 1960s.  Namely,  if the addition in the Cauchy-Davenport Theorem is restricted to distinct elements, the lower bound changes only slightly. Erd\H{o}s stated this conjecture in 1963 during a number theory conference at the University of Colorado \cite{Erdos Col. Con.}. Interestingly, Erd\H{o}s and Heilbronn did not mention the conjecture in their
$1964$ paper on sums of sets of congruence classes \cite{Erdos
Heilbronn} though Erd\H{o}s mentioned it often in his lectures (see
\cite{Nath}, page 106). Eventually the conjecture was formally
stated in Erd\H{o}s' contribution to a 1971 text \cite{Erdos} as
well as in a book by Erd\H{o}s and Graham in 1980 \cite{Erdos
Graham}. In particular,

\begin{theorem}[Erd\H{o}s-Heilbronn Problem]
If\/ $A$ and\/ $B$ are non-empty subsets of\/ $\Z/p\Z$ with $p$ prime,
then $|A\dot{+}B|\ge\min\{p,|A|+|B|-3\}$, where
$A\dot{+}B:=\{a+b\bmod p\mid a\in A$, $b\in B$ and\/ $a\ne b\,\}$.
\end{theorem}

The conjecture was first proved for the case $A=B$ by Dias da
Silva and Hamidounne in 1994~\cite{Dias da Silva} with the
more general case established by Alon, Nathanson, and
Ruzsa using the polynomial method in 1995~\cite{Alon}.
K{\'a}rolyi extended this result to abelian groups for the
case $A=B$ in 2004~\cite{Gyula2} and to cyclic groups of prime
powered order in 2005~\cite{Gyula Compact}.

A more general result of the Erd\H{o}s-Heilbronn Problem for finite groups is established in \cite{Balister/Wheeler}.

\end{document}